\documentclass{amsart}

\newtheorem{theorem}{Theorem}[section]
\newtheorem{lemma}[theorem]{Lemma}
\newtheorem{proposition}[theorem]{Proposition}
\newtheorem{corollary}[theorem]{Corollary}

\theoremstyle{definition}

\newtheorem{remark}[theorem]{Remark}

\newenvironment{proofof}[1]{\noindent{\it Proof of
#1.}}{\hfill$\square$\\\mbox{}}

\usepackage{amscd,amssymb}

\begin{document}

\title[Constructive noncommmutative invariant theory]
{Constructive noncommmutative invariant theory}

\author[M\'aty\'as Domokos and Vesselin Drensky]
{M\'aty\'as Domokos and Vesselin Drensky}
\address{MTA Alfr\'ed R\'enyi Institute of Mathematics,
Re\'altanoda utca 13-15, 1053 Budapest, Hungary}
\email{domokos.matyas@renyi.mta.hu}
\address{Institute of Mathematics and Informatics,
Bulgarian Academy of Sciences,
Acad. G. Bonchev Str., Block 8,
1113 Sofia, Bulgaria}
\email{drensky@math.bas.bg}

\thanks{This project was carried out in the framework of the exchange program
between the Hungarian and Bulgarian Academies of Sciences.
It was partially supported by the Hungarian National Research, Development and Innovation Office,  NKFIH K 119934 and
by Grant I02/18 of the Bulgarian National Science Fund.}



\begin{abstract}
The problem of finding generators of the subalgebra of invariants under the action of a group of automorphisms
of a finite dimensional Lie algebra on its universal enveloping algebra
is reduced to finding homogeneous generators of the same group acting on the symmetric tensor algebra of the Lie algebra.
This process is applied to prove a constructive Hilbert-Nagata Theorem (including degree bounds)
for the algebra of invariants in a Lie nilpotent relatively free associative algebra
endowed with an action induced by a representation of a reductive group.
\end{abstract}

\maketitle

\section{Introduction}\label{sec:intro}

Let $G$ be a subgroup of the automorphism group of a finite dimensional Lie algebra
$L$ over a field $K$ of characteristic zero. Then there is a natural induced action of $G$ on the universal enveloping algebra
$U(L)$ via associative $K$-algebra automorphisms.
First we describe a process by which one can construct generators
of the subalgebra $U(L)^G$ of $G$-invariants in $U(L)$ starting from a homogeneous generating system
of $S(L)^G$, the commutative algebra of invariants in the symmetric tensor algebra
$S(L)$ of $L$.

Namely, there is a well known $K$-vector space isomorphism $\omega:S(L)\to U(L)$ 
called the \emph{canonical bijection} in \cite[2.4.6]{dixmier}. The details  of the construction of $\omega$ 
will be given in Section~\ref{sec:enveloping}, see \eqref{eq:canonical-bijection}.  
Note that the map $\omega$ is not an algebra homomorphism. 
However, we have the following:

\begin{theorem}\label{thm:enveloping}
Suppose that $\{f_{\lambda}\mid \lambda\in \Lambda\}$ is a homogeneous generating system
of the commutative $K$-algebra $S(L)^G$, where $G$ is a subgroup
of the automorphism group of the finite dimensional Lie algebra $L$.
Then $\{\omega(f_{\lambda})\mid \lambda\in\Lambda\}$ generates the $K$-algebra $U(L)^G$.
\end{theorem}

Theorem~\ref{thm:enveloping} will be applied to noncommutative invariant theory in relatively free associative algebras.
In this paper by `algebra' we mean an associative $K$-algebra with unity, unless explicitly stated otherwise. The base field $K$ is always assumed to have characteristic zero. 
For a finite dimensional $K$-vector space $V$ denote by $T(V)$  the tensor algebra of $V$;
that is, $T(V)$ is the free algebra generated by a basis of $V$.
Given a variety $\mathfrak{R}$ of algebras we write $F(\mathfrak{R},V)$ for the
{\it relatively free algebra} in $\mathfrak{R}$ generated by a basis of $V$.
Recall that $\mathfrak{R}$ consists of all algebras that satisfy a given set of polynomial identities,
so  $F(\mathfrak{R},V)$ is the factor algebra of $T(V)$ modulo a T-ideal (an ideal stable under all algebra endomorphisms of $T(V)$).
Since we deal with algebras with unity and $K$ is infinite, $\mathfrak{R}$ necessarily contains the variety of commutative algebras,
and we have the canonical surjections
\begin{align*}\label{eq:surjections}
T(V) \twoheadrightarrow F({\mathfrak{R}},V)\twoheadrightarrow S(V) .
\end{align*}
where $S(V)$ is the symmetric tensor algebra of $V$ (i.e., the commutative polynomial algebra generated by a basis of $V$).
The above algebra surjections are $GL(V)$-equivariant. They are homomorphisms of graded algebras, where $T(V)=\bigoplus_{d=0}^\infty T^d(V)$ is endowed with the standard grading: the $d^{\mathrm{th}}$ tensor power $T^d(V)$ of $V$ is the degree $d$ homogenerous component of $T(V)$.
Given a non-zero homogeneous element $f$ in any of $T(V)$, $F({\mathfrak{R}},V)$ or $S(V)$ we shall write $\deg(f)$ for the degree of $f$.
Note that
when $\mathfrak{R}$ is the variety of all algebras we have that
$F(\mathfrak{R},V)=T(V)$, and when $\mathfrak{R}$ is the variety of commutative algebras we have $F(\mathfrak{R},V)=S(V)$.

Let $G$ be a linear algebraic group and $V$ a finite dimensional rational $G$-module;
that is, we are given a group homomorphism $\rho:G\to GL(V)$ that is a morphism of affine
algebraic varieties (from now on we shall usually omit the attribute `rational' and simply say that $V$ is a $G$-module).
The action of  $G$ on $V$
induces an action on $T(V)$, $F(\mathfrak{R},V)$, and $S(V)$ via automorphisms of graded algebras, and the above surjections are $G$-equivariant.
We are interested in the  subalgebra
\[
F(\mathfrak{R},V)^G=\{f\in F(\mathfrak{R},V)\mid g\cdot f =f\quad \text{ for all }g\in G\}
\]
of {\it $G$-invariants}. We refer to \cite{D3} and \cite{F} for surveys on results concerning subalgebras of invariants in relatively free algebras.

For an integer $p\ge 1$ denote by $\mathfrak{N}_p$ the variety of {\it Lie nilpotent algebras
of Lie nilpotency index less or equal to $p$}. In other words,  $\mathfrak{N}_p$
is the variety of algebras satisfying the polynomial identity $[x_1,\dots,x_{p+1}]=0$.
Here $[x_1,x_2]=x_1x_2-x_2x_1$, and for $i\ge 3$ we have
$[x_1,\dots,x_i]=[[x_1,\dots,x_{i-1}],x_i]$.

\begin{remark} For a (nonunitary) associative algebra nilpotence of index $\le p$ means that the algebra satisfies the polynomial identity
$x_1\cdots x_p=0$. By analogy with group theory a Lie algebra is nilpotent of index $\le p$, if it satisfies the polynomial identity
$[x_1,\dots,x_{p+1}]=0$.
\end{remark}

Our starting point is the following non-commutative generalization of the Hilbert-Nagata theorem (see for example \cite[Theorem A, p. 3]{grosshans}):

\begin{theorem}\label{thm:DD:1996} (\cite[Theorem  3.1]{DD:1996})
Suppose that $G$ is reductive, and $\mathfrak{R}\subseteq \mathfrak{N}_p$
for some $p\ge 1$. Then $F(\mathfrak{R},V)^G$ is a finitely generated algebra for any
$G$-module $V$.
\end{theorem}

\begin{remark} \label{rem:converseDD:1996}
The assumption $\mathfrak{R}\subseteq \mathfrak{N}_p$ for some $p\ge 1$ above is necessary, as the following converse of Theorem~\ref{thm:DD:1996} is
also shown in \cite{DD:1996}: If $\dim_K(V)\ge 2$ and
$F(\mathfrak{R},V)^G$ is finitely generated for all reductive subgroups $G$ of $GL(V)$, then
$\mathfrak{R}\subseteq\mathfrak{N}_p$ for some $p\ge 1$.
\end{remark}

The proof of Theorem~\ref{thm:DD:1996} in \cite{DD:1996} is non-constructive, since it  uses the Noetherian property
of $F(\mathfrak{R},V)$ for $\mathfrak{R}\subset \mathfrak{N}_p$ (proved in \cite{La}) similarly to the fundamental paper of Hilbert \cite{hilbert:1}
on the commutative case $p=1$, where the Hilbert Basis Theorem was proved.
Hilbert gave a more constructive proof of the commutative case in \cite{hilbert:2}
(explicit degree bounds for the generators were first proved by Popov \cite{popov:1}, \cite{popov:2},
and stronger bounds more recently by Derksen \cite{derksen:2}).
For constructive (commutative) invariant theory of reductive groups see also
\cite{derksen:1},  \cite{derksen-kemper:book}.

In the present paper we make Theorem~\ref{thm:DD:1996} constructive.
In algorithms for computing generators of algebras of invariants
a crucial role is played by degree bounds.
For example, an a priori degree bound for the generators of the algebra $F(\mathfrak{R},V)^G$ implies an algorithm to find explicit generators
(by solving systems of linear equations). From a different perspective,
sometimes there is a qualitatively known process that yields generators of the algebra of invariants,
and the aim is to  derive from it a degree bound for a minimal generating system,
to have some quantitative information on the algebra of invariants.

In order to discuss degree bounds we need to introduce some notation.
Given a graded algebra $\displaystyle A=\bigoplus_{d=0}^\infty A_d$ we denote by
$\beta(A)$ the minimal non-negative integer $d$ such that $A$ is generated by homogeneous elements of degree at most $d$
(and write $\beta(A)=\infty$ if there is no such $d$).
Recall that $S(V)=\bigoplus_{d=0}^\infty S^d(V)$ is graded, where the degree $d$ homogeneous component $S^d(V)$ is the 
$d^{\mathrm{th}}$ symmetric tensor power of $V$. 
Set
\[
\beta(G):=\sup\{\beta(S(V)^G)\mid V \text{ is a finite dimensional }G\text{-module}\}.
\]
It is a classical theorem of E. Noether \cite{noether} that for $G$ finite we have $\beta(G)\le |G|$.
On the other hand Derksen and Kemper \cite[Theorem 2.1]{derksen-kemper}
proved that $\beta(G)=\infty$ for any infinite group $G$. So for an infinite  group $G$,
finite degree bounds may hold only for restricted classes of $G$-modules.

Given a  set $\tau$ of isomorphism classes of simple $G$-modules, denote by
$\mathrm{add}(\tau)$ the class of all $G$-modules that are finite direct sums of simple modules whose isomorphism class belongs to $\tau$.
This definition is particularly natural when the group $G$ is reductive,
because in that case any $G$-module decomposes as a direct sum of simple $G$-modules.
By slight abuse of notation we write $V\in\tau$ if the isomorphism class of the $G$-module $V$ belongs to $\tau$.
Moreover, write $\tau^{\otimes p}$ for the isomorphism classes of simple summands
of all $G$-modules $V_1\otimes\cdots\otimes V_q$, where $V_i\in \tau$ and $q\le p$.
Set
\[
\beta_{\tau}:=\sup\{\beta(S(V)^G)\mid V\in\mathrm{add}(\tau)\}.
\]
Weyl's theorem \cite{weyl} on polarizations implies the following:

\begin{proposition}\label{prop:beta-tau} Let $G$ be a reductive group  and let
$\tau$ be a finite set of isomorphism classes of simple $G$-modules. Then the number
 $\beta_{\tau}$ is finite.
\end{proposition}
\begin{proof}
Let $V_1,\dots,V_q$ be simple $G$-modules representing the isomorphism classes in $\tau$. Let $V$ be an arbitrary $G$-module in $\mathrm{add}(\tau)$.
Then $\displaystyle V\cong \sum_{i=1}^q m_iV_i$, where
the $m_i$ are non-negative integers and $m_iV_i$ stands for the direct sum of $m_i$ copies of $V_i$.
Weyl's theorem on polarizations (the special case $h=1$  of
Theorem~\ref{thm:noncomm_weyl} below) implies that
\begin{align}\label{eq:m_iv_i}\beta(S(V)^G)\le \beta(S(\sum_{i=1}^q\min\{m_i,\dim(V_i)\}V_i)^G)\le \beta(S(\sum_{i=1}^q\dim(V_i)V_i)^G).
\end{align}
(The second inequality above follows from the fact that if $A$ is a submodule of the $G$-module $B$,
then there is a $G$-equivariant graded $K$-algebra surjection $S(B)\to S(A)$ mapping
some of the variables to zero).  Since \eqref{eq:m_iv_i} holds for any $V\in\mathrm{add}(\tau)$, we conclude the equality
\[
\beta_{\tau}=\beta(S(\sum_{i=1}^q\dim(V_i)V_i)^G).
\]
The assumption on $G$ guarantees that the number on the right hand side above is finite.
\end{proof}

Turning to a variety $\mathfrak{R}$ of associative algebras, we write
\[
\beta_{\tau,\mathfrak{R}}:=\sup\{\beta(F(\mathfrak{R},V)^G)\mid V\in\mathrm{add}(\tau)\},
\]
where
$\beta(F(\mathfrak{R},V)^G)$ is the supremum of the degrees of the elements
in a minimal homogeneous generating system of the algebra $F(\mathfrak{R},V)^G$.
Having established this notation we are in position to state the following corollary of the results of the present paper:

\begin{theorem}\label{thm:main1}
Suppose that the group $G$ is reductive and the variety $\mathfrak{R}$ is contained in $\mathfrak{N}_p$ for some $p\ge 1$.
Let $\tau$ be a finite set of isomorphism classes of simple $G$-modules. Then we have the inequality
\[
\beta_{\tau,\mathfrak{R}}\le p\beta_{\tau^{\otimes p}}.
\]
\end{theorem}

In fact the results of this paper give more: for any (not necessarily reductive) group $G\le GL(V)$  the construction of an explicit generating system of $F(\mathfrak{R},V)^G$
(where $\mathfrak{R}$ is contained in $\mathfrak{N}_p$ for some $p\ge 1$)
is reduced
to finding a generating system of a commutative algebra of $G$-invariants $S(W)^G$ for a $G$-module $W$ associated canonically to $V$.
For the precise statement see Theorem~\ref{thm:main2}.

The paper is organized as follows.
In Section~\ref{sec:enveloping} we prove Theorem~\ref{thm:enveloping},
which is then applied in Section~\ref{sec:lie-nilpotent} to prove Theorem~\ref{thm:main2}.
Finally, Theorem~\ref{thm:main1} is obtained as a consequence of Theorem~\ref{thm:main2}.

Since the technique of polarization in commutative invariant theory is fundamental for
Proposition~\ref{prop:beta-tau}, in Section~\ref{sec:weyl} we investigate to what extent it works in our noncommutative setup.
Theorem~\ref{thm:noncomm_weyl} is a noncommutative generalization of Weyl's theorem on polarization, however, it applies for
a class of varieties different from those appearing in Theorem~\ref{thm:DD:1996},
Theorem~\ref{thm:main1}, or Theorem~\ref{thm:main2}.
Section~\ref{sec:weyl} is logically independent from the previous sections.


\section{Universal enveloping algebras} \label{sec:enveloping}

Returning to the setup of the first paragraph of the Introduction, 
denote by $\pi_{U(L)}:T(L)\to U(L)$ the defining surjection onto $U(L)$ from the tensor algebra
$\displaystyle T(L)=\bigoplus_{d=0}^\infty T^d(L)$ (cf. \cite[2.1.1]{dixmier}),
and denote by $\pi_{S(L)}:T(L)\to S(L)$ the natural surjection onto the symmetric tensor algebra.
For $d=1,2,\dots$, define the linear map
\begin{equation}\label{eq:iota}
\iota_d:S^d(L)\to T^d(L),\quad \ell^d\mapsto \ell\otimes\cdots\otimes\ell \quad (\ell\in L)
\end{equation}
(or more explicitly, $\displaystyle \iota_d(\ell_1\cdots\ell_d)\mapsto \frac{1}{d!}
\sum_{\sigma\in\mathrm{Sym}_d}\ell_{\sigma(1)}\otimes \cdots \otimes \ell_{\sigma(d)}$). 
Write $\iota_0$ for the identity map of $S^0(V)=K=T^0(V)$. 

Introduce the notation
$\displaystyle U(L)_d:=\pi_{U(L)}(\bigoplus_{j=0}^dT^j(V))$, $d=0,1,2,\dots$.
Then
\[
U(L)_0\subseteq U(L)_1\subseteq U(L)_2\subseteq \cdots
\]
is the \emph{canonical filtration of $\displaystyle U(L)=\bigcup_{d=0}^{\infty}U(L)_d$} (cf. \cite[2.3.1]{dixmier}). 
Moreover, consider the linear map 
\[\omega_d:S^d(L)\to U(L) \quad \mbox{ given by } \quad \omega_d:=\pi_{U(L)}\circ\iota_d.\] 
Note that $\omega_d(S^d(L))$ is a $K$-vector space direct complement of $U(L)_{d-1}$ in $U(L)_d$ (see \cite[2.4.4. Proposition]{dixmier}). 
The direct sum 
\begin{equation}\label{eq:canonical-bijection} 
\omega:=\bigoplus_{d=0}^\infty \omega_d:S(L)=\bigoplus_{d=0}^{\infty}S^d(V)\to U(L)  
\end{equation} 
is a $K$-vector space isomorphism called the \emph{canonical bijection} in 
\cite[2.4.6]{dixmier}. Furthermore, we have the following: 

\begin{proposition}\label{prop:A-equivariant} 
The canonical  bijection $\omega:S(L)\to U(L)$  is an isomorphism of $\mathrm{Aut}(L)$-modules, where $\mathrm{Aut}(L)$ denotes the automorphism group of the Lie algebra $L$.  In particular, for any subgroup $G$ of $\mathrm{Aut}(L)$, the canonical bijection restricts to a $K$-vector space isomorphism 
\[\omega \vert_{S(L)^G}:S(L)^G\to U(L)^G.\]  
\end{proposition} 

\begin{proof} 
The universal properties of $T(L)$ and $U(L)$ imply that the action of $\mathrm{Aut}(L)$ on $L$ extends uniquely to an action via $K$-algebra automorphisms on $T(L)$ and $U(L)$, and  the algebra homomorphism  $\pi_{U(L)}$ is $\mathrm{Aut}(L)$-equivariant. 
The map $\iota_d:S^d(L)\to T^d(L)$ is even $GL(L)$-equivariant, and clearly 
$\mathrm{Aut}(L)\le GL(L)$. Therefore the composition $\omega_d=\pi_{U(L)}\circ \iota_d$ is $\mathrm{Aut}(L)$-equivariant for all $d$, implying in turn that $\omega$ is $\mathrm{Aut}(L)$-equivariant. Consequently, the $K$-linear isomorphism $\omega:S(L)\to U(L)$ is in fact an isomorphism of $G$-modules for any subgroup $G$ in $\mathrm{Aut}(L)$, and thus $\omega$ restricts to a linear isomorphism between the subspaces $S(V)^G\to U(L)^G$ of 
$G$-invariants.  
\end{proof} 

As $\omega$ is not an algebra homomorphism in general, further considerations are needed to prove Theorem~\ref{thm:enveloping}. 
Write $\eta_d:U(L)_d\to U(L)_d/U(L)_{d-1}$ for the natural surjection (with the convention
$U(L)_{-1}=0$), $\pi_{U(L),d}$ for the restriction of $\pi_{U(L)}$ to $T^d(L)$,
and $\pi_{S(L),d}$ for the restriction of $\pi_{S(L)}$ to $S^d(L)$.
The linear maps
$\eta_d,\pi_{U(L),d},\pi_{S(L),d}$ are all $\mathrm{Aut}(L)$-equivariant.
Since $\ker(\eta_d\circ \pi_{U(L),d})\supseteq \ker(\pi_{S(L),d})$
(see \cite[2.1.5. Lemma]{dixmier}), there exists a unique
$\mathrm{Aut}(L)$-module homomorphism
\[
\mu_d:S^d(L)\to U(L)_d/U(L)_{d-1}\quad \text{with} \quad \mu_d\circ \pi_{S(L),d}=\eta_d\circ\pi_{U(L),d}.
\]
Clearly $\eta_d\circ\pi_{U(L),d}\circ\iota_d\circ\pi_{S(L),d}=\eta_d\circ\pi_{U(L),d}$ (see for example \cite[2.1.5. Lemma]{dixmier}), hence  
in fact $\mu_d=\eta_d\circ\pi(U)_d\circ \iota_d=\eta_d\circ\omega_d$. 
Moreover, $\mu_d$ is a $K$-vector space isomorphism by the Poincar\'e-Birkhoff-Witt Theorem 
(or by \cite[2.4.4. Proposition]{dixmier}): 
\[
\begin{array}{ccc}
T^d(L) & \stackrel{\pi_{U(L),d}} \longrightarrow & U(L)_d \\
\downarrow{{\scriptstyle{\pi_{S(L),d}}}} &  & \downarrow{{\scriptstyle{\eta_d}}}\\
S^d(L)& \stackrel{\mu_d} \cong & U(L)_d/U(L)_{d-1} \end{array}
\]
So $\mu_d$ is an isomorphism of $\mathrm{Aut}(L)$-modules, and 
restricting to $G$-invariants (where $G$ is a subgroup of $\mathrm{Aut}(L)$) we obtain the following commutative diagram:
\begin{equation}\label{eq:diagram}
\begin{array}{ccc}
T^d(L)^G & \stackrel{\pi_{U(L),d}} \longrightarrow & U(L)_d^G \\
\downarrow{{\scriptstyle{\pi_{S(L),d}}}} &  & \downarrow{{\scriptstyle{\eta_d}}}\\
S^d(L)^G& \stackrel{\mu_d} \cong & (U(L)_d/U(L)_{d-1})^G \end{array}
\end{equation}

\begin{proofof}{Theorem~\ref{thm:enveloping}}
Since $S^d(L)^G$ is mapped isomorphically onto $(U(L)_d/U(L)_{d-1})^G$ by $\mu_d$, 
commutativity of the diagram \eqref{eq:diagram} implies 
\begin{equation}\label{eq:UU}
U(L)_d^G/U(L)_{d-1}^G=\eta_d(U(L)_d^G)=(U(L)_d/U(L)_{d-1})^G.
\end{equation}
Now let $\{f_{\lambda}\mid \lambda\in \Lambda\}$ be a homogeneous system of generators of $S(L)^G$.
The maps $\iota_d:S^d(L)\to T^d(L)$ defined by \eqref{eq:iota} are $GL(L)$-equivariant, hence they are $G$-equivariant as well. In particular,
\[
\widehat f_{\lambda}:=\iota_{\deg(f_{\lambda})}(f_{\lambda})\in T^d(L)^G \quad (\lambda\in \Lambda).
\]
Denote by $M$ the subalgebra of $T(L)$ generated by $\{\widehat f_{\lambda}\mid \lambda\in\Lambda\}$. 
By construction, $M$ is generated by homogeneous $G$-invariants, so $\displaystyle M=\bigoplus_{d=0}^{\infty}M_d$ is a graded subalgebra of $T(L)^G$.
Moreover, $E:=\pi_{U(L)}(M)$ is the subalgebra of $U(L)^G$ generated by
$\{\omega(f_{\lambda})\mid \lambda\in \Lambda\}$.  We have to show that
$E\supseteq U(L)^G$. The algebra $\displaystyle E=\bigcup_{d=0}^\infty E_d$ is filtered,  where
$\displaystyle E_d:=\pi_{U(L)}(\bigoplus_{j=0}^dM_d)$, $d=0,1,2,\dots$.
By induction on $d$ we show $E_d\supseteq U(L)_d^G$. For $d=0$ there is nothing to prove. Suppose $d>0$, and $E_{d-1}= U(L)_{d-1}^G$.
By \eqref{eq:diagram} and \eqref{eq:UU} we have the commutative diagram
\[
\begin{array}{ccc}
M_d & \stackrel{\pi_{U(L),d}} \longrightarrow & E_d \\
\downarrow{{\scriptstyle{\pi_{S(L),d}}}} &  & \downarrow{{\scriptstyle{\eta_d}}}\\
S^d(L)^G& \stackrel{\mu_d} \cong & U(L)_d^G/U(L)_{d-1}^G \end{array}
\]
The restriction of $\pi_{S(L)}$ to $M$ is a homomorphism of graded algebras 
$M\to S(L)^G$. 
By construction of $M$ the generators of $S(L)^G$ are contained in $\pi_{S(L)}(M)$. 
It follows that $\pi_{S(L)}(M)=S(L)^G$, and since $\pi_{S(L)}$ preserves the grading,  
we have $\pi_{S(L),d}(M_d)=S^d(L)^G$ for all $d$. 
It follows that 
\[\eta_d(E_d)\supseteq\eta_d(\pi_{U(L)}(M_d))=
\mu_d(\pi_{S(L),d}(M_d))=\mu_d(S^d(L)^G)=U(L)_d^G/U(L)_{d-1}^G,\] 
or, in other words,
$U(L)_d^G=E_d+U(L)_{d-1}^G$.
On the other hand, $E_d\supseteq E_{d-1}$,   and by the induction hypothesis
$E_{d-1}\supseteq U(L)_{d-1}^G$, so $U(L)_d^G=E_d+U(L)_{d-1}^G=E_d$.
We conclude  $\displaystyle E=\bigcup_{d=0}^{\infty} E_d=\bigcup_{d=0}^{\infty}U(L)_d^G=U(L)^G$.
\end{proofof} 

In the above proof we pointed out that the natural surjection $\eta_d:U(L)_d\to U(L)_d/U(L)_{d-1}$ maps $U(L)_d^G$ onto $(U(L)_d/U(L)_{d-1})^G$ (see \eqref{eq:UU}). 
For later use we state a general lemma guaranteeing that a surjective map 
restricts to a surjection between the corresponding subspaces of invariants. 

Given a reductive group $H$, recall that by a \emph{rational $H$-module}
we mean a not necessarily finite dimensional vector space $X$ together with an action of $H$ via linear transformations,
such that any $x\in X$ is contained in a finite dimensional
subspace $Y$ of $X$ which is stable under the action of $H$, i.e., $H\cdot Y=Y$,
and the group homomorphism $H\to GL(Y)$ giving the action of $H$ on $Y$ is a morphism of affine algebraic varieties.
We shall denote by $X^H$ the subspace of $H$-invariants in $X$.

\begin{lemma}\label{lemma:surjective}
Let $H$ be a reductive group and $\varphi:X\to Z$ a surjective homomorphism of rational
$H$-modules. Then for any subgroup $G$ of $H$ we have $Z^G=\varphi(X^G)$.
\end{lemma}

\begin{proof}
Since $H$ acts completely reducibly on $X$, the $H$-submodule $\ker(\varphi)$ has
an  $H$-module direct complement $Y$ in $X$, so $X=\ker(\varphi)\oplus Y$.
The restriction $\varphi\vert_Y:Y\to Z$ is an $H$-module isomorphism, hence it is a $G$-module isomorphism as well,
thus it restricts to a $K$-vector space isomorphism
$Y^G\to Z^G$. Clearly $Y^G\subseteq X^G$, so we derive $\varphi(X^G)\supseteq\varphi(Y^G)=Z^G$.
The reverse inclusion $\varphi(X^G)\subseteq Z^G$ holds for any $G$-module homomorphism.
\end{proof}


\section{Lie nilpotent relatively free algebras}\label{sec:lie-nilpotent}

Fix an integer $p\ge 1$, a variety $\mathfrak{R}$
contained in $\mathfrak{N}_p$, and a finite dimensional rational $G$-module $V$ (where $G$ is an arbitrary linear algebraic group).
Without loss of generality we may assume that $G\le GL(V)$ (i.e., $V$ is a faithful $G$-module). Set $F:=F(\mathfrak{R},V)$.
Consider the subspaces in $T(V)$ given by
\[
V^{[d]}:=\mathrm{Span}_K\{[v_1,\dots,v_d]\mid v_1,\dots,v_d\in V\}
\]
for $d=1,2,\dots$, and $V^{[1]}:=V$.
As a consequence of the Jacobi identity we have
\begin{equation}\label{eq:jacobi}
[V^{[d]},V^{[e]}]\subseteq V^{[d+e]}.
\end{equation}
The subspace $V^{[d]}$ is a $GL(V)$-submodule in the $d^{\mathrm{th}}$ tensor power $V^{\otimes d}$ of $V$. Set
\[
V^{[\le p]}:=\bigoplus_{d=1}^p V^{[d]},
\]
and let  $T:=T(V^{[\le p]})$ be the tensor algebra generated by $V^{[\le p]}$.
Denote by $\varepsilon^i$ the $i^{\mathrm{th}}$ standard basis vector
$\varepsilon^i:=(0,\dots,0,1,0,\dots,0)\in\mathbb{N}_0^p$ (the coordinate $1$ is in the   $i^{\mathrm{th}}$ position). The algebra $T$
has an $\mathbb{N}_0^p$-grading
\begin{equation}\label{eq:multigrading}
T=\bigoplus_{\alpha\in\mathbb{N}_0^p} T_{\alpha},
\qquad T_{\varepsilon^i}=V^{[i]} \subseteq V^{[\le p]}
\text{ for }i=1,\dots,p.
\end{equation}
Note that $GL(V)$ acts on $T$ via $\mathbb{N}_0^p$-graded algebra automorphisms.
The symmetric tensor algebra $S:=S(V^{[\le p]})$ is endowed with the analogous $\mathbb{N}_0^p$-grading,
so the natural surjection $\pi_{S(L)}:T\to S$ is a $GL(V)$-equivariant homomorphism of $\mathbb{N}_0^p$-graded algebras.

The tautological linear embedding of $V^{[\le p]}$  into $T(V)$ extends to a $GL(V)$-equivariant algebra surjection $\pi:T\to T(V)$.
Composing this with the natural surjection $\nu:T(V)\to F(\mathfrak{R},V)$ we obtain
\begin{equation}\label{eq:pi_F}
\pi_F:=\nu\circ\pi:T\to F.
\end{equation}
A homogeneous element in $T$ may not be mapped to a homogeneous element of $T(V)$ by $\pi$ (or of $F$ by $\pi_F$).
However, the multihomogeneous component $T_{\alpha}\subset T$
is mapped under $\pi$ (respectively $\pi_F$) into the homogeneous component of $T(V)$
(respectively $F$) of degree $\displaystyle \sum_{i=1}^pi\alpha_i$.

For $\alpha\in\mathbb{N}_0^p$ write $\displaystyle |\alpha|:=\sum_{i=1}^p\alpha_i$.
Denote by $\iota_{\alpha}:S_{\alpha}\to T_{\alpha}$ the linear map given by
\[
\iota_{\alpha}(v_1\cdots v_{|\alpha|})=\frac{1}{|\alpha|!}\sum_{\sigma\in {\mathrm{Sym}}_{|\alpha|}}
v_{\sigma(1)}\otimes \cdots \otimes v_{\sigma(|\alpha|)}
\]
(where ${\mathrm{Sym}}_d$ stands for the symmetric group on $\{1,\dots,d\}$ and
$v_1,\dots, v_{|\alpha|}\in V^{[\le p]}$).
Obviously $\iota_{\alpha}$ is $GL(V)$-equivariant (as it is even $GL(V^{[\le p]})$-equivariant),
and hence $\iota_{\alpha}$ is $G$-equivariant (just like $\pi_{S(L)}$).
Thus $\iota_{\alpha}(S_{\alpha}^G)\subseteq T_{\alpha}^G$. Moreover,
$\pi_{S(L)}\circ \iota_{\alpha}$ is the identity map on $S_{\alpha}$,
consequently
$\pi_{S(L)}(\iota_{\alpha}(S_{\alpha}^G))=S_{\alpha}^G$. In particular,
$\pi_{S(L)}(T_{\alpha}^G)=S_{\alpha}^G$.

\begin{theorem}\label{thm:main2}
Let $\mathfrak{R}$ be a variety contained in $\mathfrak{N}_p$ for some $p\ge 1$.
Let $V$ be a $G$-module,  and $S$, $T$, $F$ the algebras defined above.
Take a multihomogeneous (with respect to the $\mathbb{N}_0^p$-grading \eqref{eq:multigrading})
$K$-algebra generating system $\{f_{\lambda}\mid \lambda\in \Lambda\}$  of $S^G$, denote by $\alpha_{\lambda}$  the multidegree of $f_{\lambda}$,  and set
\[
\widehat f_{\lambda}:=\iota_{\alpha_{\lambda}}(f_{\lambda})\in T_{\alpha_{\lambda}}^G
 \quad (\lambda\in\Lambda).\]
Then the $K$-algebra $F^G$ is generated by the homogeneous elements
$\{\pi_F(\widehat f_{\lambda})\mid \lambda\in\Lambda\}$.
Furthermore,
\[
\deg(\pi_F(\widehat f_{\lambda}))=\displaystyle \sum_{i=1}^pi\alpha_{\lambda,i}.
\]
\end{theorem}

\begin{remark}
The algebra $S^G$ is known to be finitely generated (in addition to the case when $G$ is reductive)
also when  $G$ is a maximal unipotent subgroup of a reductive group (see \cite{hadziev} or \cite[Theorem 9.4]{grosshans}),
and consequently when $G$  is a Borel subgroup of a reductive group.
In these cases Theorem~\ref{thm:main2} reduces the construction of a finite generating system of $F^G$
to the construction of a finite generating system in the commutative algebra of invariants $S^G$.
\end{remark}

\begin{proofof}{Theorem~\ref{thm:main2}}
The subspace $\displaystyle \bigoplus_{d=1}^\infty V^{[d]}$ of $T(V)$ is a Lie subalgebra
(in fact it is the free Lie algebra generated by $\dim_K(V)$ elements).
It contains the Lie ideal $\displaystyle \bigoplus_{d>p}V^{[d]}$, and the corresponding factor
Lie algebra $L_p(V)$ is the relatively free nilpotent Lie algebra of nilpotency index $p$ generated by $\dim_K(V)$ elements.
We identify the underlying vector space of $L_p(V)$ with $\displaystyle \bigoplus_{d=1}^pV^{[d]}$ in the obvious way.
Then $T$ is identified with $T(L_p(V))$ and $S$ is identified with $S(L_p(V))$. 
Applying Theorem~\ref{thm:enveloping} for the Lie algebra $L_p(V)$ and the group $G\le GL(V)$  (note that $GL(V)$ is naturally a subgroup of the automorphism group 
$\mathrm{Aut}(L_p(V))$ of the Lie algebra $L_p(V)$) we obtain that
\begin{equation}\label{eq:widehat}
\{\pi_{U(L_p(V))}(\widehat f_{\lambda})\mid\lambda\in \Lambda\}
\end{equation}
is a generating system of $U(L_p(V))^G$.
By relative freeness of $L_p(V)$ in the variety $\mathfrak{N}_p$, the identity map $V\to V$
extends to a (surjective) Lie algebra homomorphism from $L_p(V)$ onto the Lie subalgebra of $F$ generated by $V$.
Now the universal property of the universal enveloping algebra implies that this Lie algebra homomorphism extends
to an associative algebra homomorphism $\gamma:U(L_p(V))\to F$. 
The image of $\gamma$ contains the $K$-algebra generating system $V\subset F$, hence $\gamma$ is surjective onto $F$. It is also $GL(V)$-equivariant,
therefore by Lemma~\ref{lemma:surjective} we have
$\gamma(U(L_p(V))^G)=F^G$.
It follows that the generating system \eqref{eq:widehat} of $U(L_p(V))^G$ is mapped by
$\gamma$ to a generating system of $F^G$.
On the other hand, by construction $\pi_F$ and $\gamma\circ \pi_{U(L_p(V))}$ agree 
on the subspace $V^{[\le p]}=L_p(V)$ of $T$ generating $T$ as a $K$-algebra, hence 
\[
\pi_F=\gamma\circ \pi_{U(L_p(V))}:T=T(L_p(V))\to F.
\]
Thus $\pi_F(\widehat f_{\lambda})=\gamma(\pi_{U(L_p(V))}(\widehat f_{\lambda}))$, and these elements generate $F^G$ as $\lambda$ varies over $\Lambda$.
\end{proofof}

\begin{proofof}{Theorem~\ref{thm:main1}}
Let $V$ be a $G$-module in $\mathrm{add}(\tau)$. Now $G$ is reductive, so we may take a finite minimal multihomogeneous
(with respect to the $\mathbb{N}_0^p$-grading \eqref{eq:multigrading}) generating system $\{f_{\lambda}\mid \lambda \in \Lambda\}$
of $S(V^{[\le p]})^G$.
Consider the
generating system $\{\pi_F(\widehat f_{\lambda})\mid \lambda\in\Lambda\}$ of $F(\mathfrak{R},V)^G$ given by Theorem~\ref{thm:main2}.
Note that $\{f_{\lambda}\mid \lambda \in \Lambda\}$ is a minimal homogeneous generating system of $S^G$ with respect to the standard grading,
where $V^{[\le p]}$ is the degree $1$ homogeneous component of $S$.
Since $G$ is reductive, all $G$-modules are completely reducible. Moreover, $V^{[q]}$ is a $GL(V)$-module
(hence $G$-module) direct summand in $V^{\otimes q}$. It follows that
$V^{[\le p]}$ belongs to
$\mathrm{add}(\tau^{\otimes p})$, implying that $\deg(f_\lambda)=\sum_{i=1}^p\alpha_{\lambda,i}\le\beta_{\tau^{\otimes p}}$ for each $\lambda\in \Lambda$
(recall that the multidegree of $f_{\lambda}$ is
$\alpha_{\lambda}=(\alpha_{\lambda,1},\dots,\alpha_{\lambda,p})$, and $\deg(f_{\lambda})$ here stands for the degree of $f_{\lambda}$ with respect to the standard grading on $S$, for which the generators have degree $1$).
So we have
\[
\deg(\pi_F(\widehat f_{\lambda}))=\sum_{i=1}^pi\alpha_{\lambda,i}\le p\sum_{i=1}^p\alpha_{\lambda,i}
=p\deg(f_{\lambda})\le p\beta_{\tau^{\otimes p}}.
\]
In other words,  $F(\mathfrak{R},V)^G$ is generated by elements of degree at most
$p\beta_{\tau^{\otimes p}}$. This holds for any $G$-module $V\in\mathrm{add}(\tau)$, therefore the desired inequality $\beta_{\tau,\mathfrak{R}}\le p\beta_{\tau^{\otimes p}}$ holds.
\end{proofof}

\begin{remark} In the proof of Theorem~\ref{thm:main1} we used that all simple $G$-submodules of $V^{[q]}$ are contained in $V^{\otimes  q}$, because
$V^{[q]}$ is a $GL(V)$-submodule of $V^{\otimes q}$. We mention that it was proved by
Klyachko
\cite{klyachko} that for $q\neq 4,6$, all simple $GL(V)$-submodules of $V^{\otimes q}$ different
from the $q^{\mathrm{th}}$ exterior or symmetric powers of $V$ are contained in $V^{[q]}$.
\end{remark}

Following \cite{domokos-drensky:2016} for a finite group $G$ we set
\[
\beta(G,\mathfrak{R})=\sup\{\beta(F(\mathfrak{R},V)^G\mid V\text{ is a }G\text{-module}\}.
\]

\begin{corollary}\label{cor:finite} For a finite group $G$ we have
\[
\beta(G,\mathfrak{N}_p)\le p\beta(G).
\]
\end{corollary}

\begin{remark}
The results of \cite{domokos-drensky:2016} provide an upper bound for
$\beta(G,\mathfrak{N}_p)$ of different nature.
\end{remark}


\section{A noncommutative generalization of Weyl's theorem}\label{sec:weyl}

In this section we assume that the variety $\mathfrak{R}$ is generated (in the sense of universal algebra) by a finitely generated algebra.
Kemer \cite{kemer} proved that then
$\mathfrak{R}$ satisfies a Capelli identity. This means that there exists a positive integer
$h=h(\mathfrak{R})$ such that all algebras in $\mathfrak{R}$ satisfy the polynomial identity
\begin{align}\label{eq:capelli}
\sum_{\pi\in \mathrm{Sym}_{h+1}} (-1)^{\pi}x_{\pi(1)}y_1x_{\pi(2)}y_2\cdots y_hx_{\pi(h+1)}=0.
\end{align}

Our focus is on the case when $V=U+mW$, where
\[
mW=W+\cdots+W
\]
is the direct sum of $m$ copies of a $G$-module $W$, and $U$ is a $G$-module.
We shall identify $mW$ with $W\otimes K^m$, which is naturally a $GL(W)\times GL(K^m)$-module
(here $K^m$ stands for the space of column vectors of length $m$ over $K$).
Also $V=U+mW$ is naturally a $GL(U)\times GL(W)\times GL(K^m)$-module.
Moreover, for $l\le m$ we shall identify $K^l$ with the subspace of $K^m$ consisting of column vectors whose last $m-l$ coordinates are zero.
Accordingly $lW$ is identified with the subspace of $mW$ consisting of $m$-tuples of elements of $W$ with the zero vector as the last $m-l$ component.
In this way $F(\mathfrak{R},U+lW)$ becomes a subalgebra of
$F(\mathfrak{R},U+mW)$.

\begin{theorem}\label{thm:noncomm_weyl}
Let $G$ be a group, let $U$ and $W$ be $G$-modules, $n=\dim(W)$, and let $B$ be a set of generators of the algebra $F(\mathfrak{R},U+nhW)^G$,
where $h=h(\mathfrak{R})$ as above.
Then for any $m\ge nh$ the algebra $F(\mathfrak{R},U+mW)^G$ is generated by
\[
\{g\cdot f\mid g\in GL(K^m), \ f\in B\}.
\]
\end{theorem}

Recall that the relatively free algebra $F(\mathfrak{R},V)$ is graded such that the subspace $V$ is the degree $1$ homogeneous component.

\begin{corollary}\label{cor:beta}
In the scenario of Theorem~\ref{thm:noncomm_weyl} for all positive integers $m$ we
have the inequality
\[
\beta(F(\mathfrak{R},U+mW)^G)\le \beta(F(\mathfrak{R},U+nhW)^G).
\]
\end{corollary}

\begin{remark}\label{remark:weyl}
In the special case when $\mathfrak{R}$ is the variety of commutative algebras
we have $h(\mathfrak{R})=1$, and hence Theorem~\ref{thm:noncomm_weyl} in this special case recovers Weyl's theorem on polarization, which is a main theme in
\cite{weyl}.
\end{remark}

In order to prove Theorem~\ref{thm:noncomm_weyl} we need to recall some facts about polynomial representations of the general linear group.
A {\it partition} $\lambda$ of $d$ (we write $\lambda\vdash d$) is a finite non-increasing sequence
$\lambda=(\lambda_1,\lambda_2,\dots)$ of non-negative integers
(with the convention that zeros can be freely appended to or removed from the end) such that $\displaystyle \sum_i{\lambda_i}=d$.
Write
$\mathrm{ht}(\lambda)$ for the number of non-zero components of $\lambda$.
Denote by $S^{\lambda}(.)$ the Schur functor (see for example \cite{procesi}), so
for  a finite dimensional $K$-vector space $E$ and a partition $\lambda$ with $\mathrm{ht}(\lambda)\le \dim(E)$, we have that
$S^{\lambda}(E)$ is a simple polynomial $GL(E)$-module.

\begin{lemma}\label{lemma:pieri}
If the simple $GL(E)$-module $S^{\mu}(E)$ occurs as a direct summand in $S^{\lambda}(U+nE)$,
where $n$ is a positive integer and $GL(E)$ acts trivially on the finite dimensional $K$-vector space $U$,
then $\mathrm{ht}(\mu)\le n\cdot \mathrm{ht}(\lambda)$.
\end{lemma}

\begin{proof}  Suppose that
$S^{\mu}(E)$ occurs as a direct summand in $S^{\lambda}(V)$, where
$V=U+nE$.
It follows from Pieri's rules (see \cite{macdonald}) that
\begin{align}\label{eq:summand_1}
S^{\lambda}(V)\text{ is a direct summand in }\bigotimes_{i=1}^{\mathrm{ht}(\lambda)}S^{(\lambda_i)}(V),
\end{align}
where for the partition $(d)$ with only one non-zero part, $S^{(d)}(V)$ is the $d^{\mathrm{th}}$ symmetric tensor power of $V$.
Therefore $S^{\mu}(E)$ is a direct summand in
$\displaystyle \bigotimes_{i=1}^{\mathrm{ht}(\lambda)}S^{(\lambda_i)}(U+nE)$. Any simple $GL(E)$-module summand of
$S^{(\lambda_i)}(U+nE)$ is isomorphic to $S^{\nu^i}(E)$ for some partition $\nu^i=(\nu^i_1,\dots,\nu^i_n)$ with
$\mathrm{ht}(\nu_i)\le n$ by the Cauchy identity (see \cite{procesi}).
Therefore again by \eqref{eq:summand_1},  $S^{\mu}(E)$ is a direct summand in
$S^{\nu^1}(E)\otimes \cdots \otimes S^{\nu^{\mathrm{ht}(\lambda)}}(E)$,
hence $S^{\mu}(E)$ is a direct summand in
\[
\bigotimes_{i=1}^{\mathrm{ht}(\lambda)}\bigotimes_{j=1}^nS^{(\nu^i_j)}(E).
\]
We conclude by Pieri's rules that $\mathrm{ht}(\mu)$ is bounded by the number $n\mathrm{ht}(\lambda)$ of tensor factors in the above expression.
\end{proof}

\begin{proofof}{Theorem~\ref{thm:noncomm_weyl}}
Set $V=U+mW=U+W\otimes K^m$. Since $F(\mathfrak{R},V)$ satisfies the Capelli identity
\eqref{eq:capelli}, by a theorem of Regev \cite{regev} (see also \cite[Theorem 2.3.4]{drensky_formanek})
no $S^{\lambda}(V)$ with $\mathrm{ht}(\lambda)>h=h(\mathcal{R})$ occurs
as a summand of $F(\mathfrak{R},V)$. So the degree $d$ homogeneous component of
$F(\mathfrak{R},V)$ has the $GL(V)$-module decomposition
\[
F(\mathfrak{R},V)_d \cong
\sum_{\lambda\vdash d \atop \mathrm{ht}(\lambda)\le h}m^{\lambda} S^{\lambda}(V)
\]
with some non-negative integers $m^{\lambda}$.
Setting $E=K^m$ and $n=\dim(W)$,
Lemma~\ref{lemma:pieri} shows that as a $GL(E)$-module,
$F(\mathfrak{R},U+W\otimes E)_d$ decomposes as
\begin{align} \label{eq:summand_2}
F(\mathfrak{R},U+nE)_d \cong\sum_{\mu\vdash d \atop \mathrm{ht}(\mu)\le nh}r^{\mu}S^{\mu}(E)
\end{align}
with some non-negative integers $r^{\mu}$.
Since the actions of $G$ and $GL(E)$ commute, the algebra $F(\mathfrak{R},U+W\otimes E)^G$ is a $GL(E)$-submodule in $F(\mathfrak{R},U+W\otimes E)$.
Thus each simple $GL(E)$-module direct summand of $F(\mathfrak{R},U+W\otimes E)^G$ is isomorphic to
$S^{\mu}(E)$ for some partition $\mu$ with $\mathrm{ht}(\mu)\le nh$. Such a summand is generated by a highest weight vector $w$,
having the property that for an element
$g=\mathrm{diag}(z_1,\dots,z_m)\in GL(K^m)=GL(E)$ we have
$g\cdot w=(z_1^{\mu_1}\cdots z_m^{\mu_m})w$. Since $\mu_{nh+1}=\cdots=\mu_m=0$, we conclude that $w$ belongs
to the subalgebra $F(\mathfrak{R},U+W\otimes K^{nh})=
F(\mathfrak{R},U+nhW)$ of $F(\mathfrak{R},U+W\otimes E)$.
Thus $F(\mathfrak{R},U+W\otimes E)^G$ is contained in the $GL(E)$-submodule of
$F(\mathfrak{R},U+W\otimes E)$ generated by $F(\mathfrak{R},U+W\otimes K^{nh})^G$.
Now $\{g\cdot f\mid g\in GL(K^m), \ f\in B\}$ spans a $GL(E)$-submodule, hence it generates a $GL(E)$-stable subalgebra of $F(\mathfrak{R},U+W\otimes E)^G$.
By construction this subalgebra contains $F(\mathfrak{R},U+W\otimes K^{nh})^G$, so
it contains $F(\mathfrak{R},U+W\otimes E)^G$.
\end{proofof}


\section*{Acknowledgements}

We are very grateful to the referees for their suggestions. In particular,
it was the referee's insight that our results -- originally dealing only with relatively free algebras -- should be treated
as a consequence of a statement on universal enveloping algebras.


\end{document}